\documentclass[12pt,a4paper]{amsart} 

\usepackage{hyperref} 
\usepackage{amsrefs}
\usepackage{amssymb}
\usepackage[all]{xy}

\let\mod=\undefined
\DeclareMathOperator{\GL}{GL} %
\DeclareMathOperator{\id}{id} %
\DeclareMathOperator{\pd}{pd} %
\DeclareMathOperator{\add}{add} %
\DeclareMathOperator{\End}{End} %
\DeclareMathOperator{\Ext}{Ext} %
\DeclareMathOperator{\Hom}{Hom} %
\DeclareMathOperator{\mod}{mod} %
\DeclareMathOperator{\rad}{rad} %
\DeclareMathOperator{\rep}{rep} %
\DeclareMathOperator{\gldim}{gl.dim} %

\newcommand{\bd}{\mathbf{d}} %

\DeclareMathOperator{\bdim}{\mathbf{dim}} %

\newcommand{\calC}{\mathcal{C}} %
\newcommand{\calE}{\mathcal{E}} %
\newcommand{\calL}{\mathcal{L}} %
\newcommand{\calO}{\mathcal{O}} %
\newcommand{\calP}{\mathcal{P}} %
\newcommand{\calR}{\mathcal{R}} %
\newcommand{\calU}{\mathcal{U}} %
\newcommand{\calV}{\mathcal{V}} %
\newcommand{\calX}{\mathcal{X}} %

\newcommand{\bbM}{\mathbb{M}} %
\newcommand{\bbN}{\mathbb{N}} %
\newcommand{\bbZ}{\mathbb{Z}} %

\newcommand{\frakR}{\mathfrak{R}} %

\newcommand{\ol}{\overline}

\newtheorem*{maintheorem}{Main Theorem}

\numberwithin{equation}{section}
\renewcommand{\theequation}{\thesection.\arabic{equation}}

\newtheorem{corollary}[equation]{Corollary}
\newtheorem{lemma}[equation]{Lemma}
\newtheorem{proposition}[equation]{Proposition}
\newtheorem{theorem}[equation]{Theorem}

\title[Canonical tilting modules over shod algebras]%
{Canonical tilting modules over shod algebras are regular in
codimension one}

\author{Grzegorz Bobi{\'n}ski}

\address{Faculty of Mathematics and Computer Science \\ Nicolaus
Copernicus University \\ ul.~Chopina 12/18 \\ 87-100 Toru\'n \\
Poland}

\email{gregbob@mat.uni.torun.pl}

\date{}



\begin{document}

\begin{abstract}
We show that for a class of modules over shod algebras, including
the canonical tilting modules, the closures of the corresponding
orbits in module varieties are regular in codimension one.
\end{abstract}

\maketitle

Throughout the paper $k$ is a fixed algebraically closed field. By
$\bbZ$, $\bbN$, and $\bbN_+$, we denote the sets of integers,
nonnegative integers, and positive integers, respectively. If $i,
j \in \bbZ$, then $[i, j]$ denotes the set of all $l \in \bbZ$
such that $i \leq l \leq j$.

\section*{Introduction and the main result}

Given a finite dimensional $k$-algebra $\Lambda$ and an element
$\bd$ of the Grothendieck group $K_0 (\Lambda)$ of the category of
$\Lambda$-modules, one defines the variety $\mod_\Lambda (\bd)$ of
$\Lambda$-modules of dimension vector $\bd$. A product $\GL (\bd)$
of general linear groups acts on $\mod_\Lambda (\bd)$ in such a
way that the $\GL (\bd)$-orbits correspond to the isomorphism
classes of $\Lambda$-modules of dimension vector $\bd$. The study
of properties of the module varieties is an important direction of
research in the representation theory of algebras (see for
example~\cites{Crawley-BoeveySchroer2002, GeissSchroer2003,
Richmond2001, Schroer2004}). It is particularly interesting how
properties of points of $\mod_\Lambda (\bd)$, viewed as
$\Lambda$-modules, correspond to their geometric properties. For
example, if $M \in \mod_\Lambda (\bd)$, then $\End_\Lambda (M)$
measures the size of its $\GL (\bd)$-orbit, $\Ext_\Lambda^1 (M,
M)$ measures the size of the tangent space to $\mod_\Lambda (\bd)$
at $M$~\cite{Kraft1984}, and, roughly speaking, $\Ext_\Lambda^2
(M, M)$ says how far is $M$ from being a regular point of
$\mod_\Lambda (\bd)$~\cite{Geiss1996a}. Using these observations a
series of results describing properties of varieties of module
over quasitilted algebras was obtained (see for
example~\cites{BarotSchroer2001, Bobinski2002, Bobinski2008a,
BobinskiSkowronski1999a, BobinskiSkowronski1999b,
BobinskiSkowronski2002}). Recall, that an algebra $\Lambda$ is
called quasi-tilted if $\gldim \Lambda \leq 2$ and either
$\pd_\Lambda X \leq 1$ or $\id_\Lambda X \leq 1$ for each
indecomposable $\Lambda$-module $X$.

The shod (small homological dimensions) algebras introduced by
Coelho and Lanzilotta~\cite{CoelhoLanzilotta1999} are natural
generalization of the quasitilted algebras. An algebra $\Lambda$
is called shod if either $\pd_\Lambda X \leq 1$ or $\id_\Lambda X
\leq 1$ for each indecomposable $\Lambda$-module $X$. The shod
algebras share an important property with the quasitilted
algebras, namely for each indecomposable module $X$ over a shod
algebra $\Lambda$, either $X \in \calL_\Lambda$ or $X \in
\calR_\Lambda$, where $\calL_\Lambda$ denotes the class of all
indecomposable $\Lambda$-modules $X$ such that $\pd_\Lambda Y \leq
1$ for each predecessor $Y$ of $X$ in the module category, and
$\calR_\Lambda$ is defined dually. The aim of the paper is to make
a first step in order to generalize geometric results obtained for
quasititled algebras to shod algebras. Obviously, we may
concentrate on strict shod algebras, where a shod algebra
$\Lambda$ is called strict if $\gldim \Lambda > 2$ (equivalently,
$\gldim \Lambda = 3$).

Given a strict shod algebra $\Lambda$, Coelho, Happel and
Unger~\cite{CoelhoHappelUnger2002} defined the canonical tilting
module $T$, which is the direct sum of the indecomposable
$\Ext$-injective objects in $\calL_\Lambda$ and the indecomposable
projective $\Lambda$-modules which do not belong to
$\calL_\Lambda$. Such a module seems to be not far from being
directing (in the sense of~\cite{HappelRingel1993}). Consequently,
the following theorem, which is the main result of the paper,
should be viewed as an extension
of~\cite{Bobinski2007preprint}*{Main Theorem} to the case of shod
algebras.

\begin{maintheorem}
Let $\Lambda$ be a strict shod algebra and $T$ the canonical
tilting module over $\Lambda$. If $M$ is a direct summand of $T^n$
for some $n \in \bbN$, then the closure of the orbit of $M$ is
regular in codimension one.
\end{maintheorem}

We remark that in the situation of the theorem the closure of the
orbit of $M$ is an irreducible component of the corresponding
module variety.

The paper is organized as follows. In the first section we recall
basic informations about quivers and their representations. We
also collect necessary facts about shod algebras there. In the
next section, we define module varieties and present geometric
tools used in the proof of Main Theorem. In the final section, we
prove Main Theorem.

For a basic background on representation theory of algebras we
refer to~\cite{AssemSimsonSkowronski2006}. Basic algebraic
geometry used in the article can be found in~\cite{Kunz1985}.

The article was written while the author was staying at
University of Bielefeld as an Alexander von Humboldt Foundation
fellow.

\section{Preliminaries on quivers and shod algebras}

By a quiver $\Delta$ we mean a finite set $\Delta_0$ of vertices
and a finite set $\Delta_1$ of arrows together with maps $s, t :
\Delta_1 \to \Delta_0$ which assign to $\alpha \in \Delta_1$ the
starting vertex $s_\alpha$ and the terminating vertex $t_\alpha$.
By a path of length $n \in \bbN_+$ we mean a sequence $\sigma =
\alpha_1 \cdots \alpha_n$ with $\alpha_1, \ldots, \alpha_n \in
\Delta_1$ such that $s_{\alpha_i} = t_{\alpha_{i + 1}}$ for each
$i \in [1, n - 1]$. In the above situation we put $s_\sigma =
s_{\alpha_n}$ and $t_\sigma = t_{\alpha_1}$. Additionally, for
each $x \in \Delta_0$ we introduce the trivial path of length $0$
starting and terminating at $x$ and denoted by $x$. A path
$\sigma$ of positive length such that $s_\sigma = t_\sigma$ is
called an oriented cycle. A full subquiver $\Delta'$ of $\Delta$
(i.e.\ a pair $(\Delta_0', \Delta_1')$ such that $\Delta_0'
\subseteq \Delta_0$ and $\Delta_1'$ consists of all $\alpha \in
\Delta_1$ such that $s_\alpha, t_\alpha \in \Delta_0'$) is called
convex if for every path $\alpha_1 \cdots \alpha_n$ in $\Delta$
with $\alpha_1, \ldots, \alpha_n \in \Delta_1$ and $s_{\alpha_n},
t_{\alpha_1} \in \Delta_0'$, $\alpha_i \in \Delta_1'$ for each $i
\in [1, n]$.

For a quiver $\Delta$ we denote by $k \Delta$ the path algebra of
$\Delta$ defined as follows. The elements of $k \Delta$ are the
formal linear combinations of paths in $\Delta$ and for two paths
$\sigma_1$ and $\sigma_2$ the product of $\sigma_1$ and $\sigma_2$
is either the composition $\sigma_1 \sigma_2$ of paths, if
$s_{\sigma_1} = t_{\sigma_2}$, or $0$, otherwise. Fix $x, y \in
\Delta_0$ and let $\rho = \lambda_1 \sigma_1 + \cdots + \lambda_n
\sigma_n$ for $n \in \bbN_+$, $\lambda_1, \ldots, \lambda_n \in k
\setminus \{ 0 \}$, and pairwise different paths $\sigma_1$,
\ldots, $\sigma_n$. If $s_{\sigma_i} = x$ and $t_{\sigma_i} = y$
for each $i \in [1, n]$, then we write $s_{\rho} = x$ and
$t_{\rho} = y$. If, additionally, the length of $\sigma_i$ is
bigger than $1$ for each $i \in [1, n]$, then $\rho$ is called a
relation. A set $\frakR$ of relations is called minimal if for
each $\rho \in \frakR$, $\rho$ does not belong to the ideal
$\langle \frakR \setminus \{ \rho \} \rangle$ generated by $\frakR
\setminus \{ \rho \}$. If $\frakR$ is a minimal set of relations
such that there exists $n \in \bbN_+$ with the property $\sigma
\in \langle \frakR \rangle$ for each path $\sigma$ of length at
least $n$, then the pair $(\Delta, \frakR)$ is called a bound
quiver. If $(\Delta, \frakR)$ is a bound quiver, then $k \Delta /
\langle \frakR \rangle$ is called the path algebra of $(\Delta,
\frakR)$.

Gabriel proved (see for
example~\cite{AssemSimsonSkowronski2006}*{Corollaries~I.6.10
and~II.3.7}) that for each finite dimensional algebra $\Lambda$
there exists a bound quiver $(\Delta, \frakR)$ such that the
category $\mod \Lambda$ of $\Lambda$-modules is equivalent to the
category of modules over the path algebra of $(\Delta, \frakR)$.
In addition, $\Delta$ is uniquely (up to isomorphism) determined
by $\Lambda$ and we call it the Gabriel quiver of $\Lambda$. An
algebra $\Lambda$ is called triangular if its Gabriel quiver
contains no oriented cycles. From now on we assume that all
considered algebras are path algebras of bound quivers.

Let $\Delta$ be a quiver. A collection $M = (M_x, M_\alpha)_{x \in
\Delta_0, \alpha \in \Delta_1}$ of finite dimensional vector
spaces $M_x$, $x \in \Delta_0$, and linear maps $M_\alpha :
M_{s_\alpha} \to M_{t_\alpha}$, $\alpha \in \Delta_1$, is called a
representation of $\Delta$. If $M$ and $N$ are representations of
$\Delta$, then the morphism space $\Hom_\Delta (M, N)$ consists of
the collections $f = (f_x)_{x \in \Delta_0}$ of linear maps $f_x :
M_x \to N_x$, $x \in \Delta_0$, such that $N_{\alpha} f_{s_\alpha}
= f_{t_\alpha} M_\alpha$ for each $\alpha \in \Delta_1$. If
$\sigma = \alpha_1 \cdots \alpha_n$, for $n \in \bbN_+$ and
$\alpha_1, \ldots, \alpha_n \in \Delta_1$, is a path, then for a
representation $M$ of $\Delta$ we put $M_\sigma = M_{\alpha_1}
\cdots M_{\alpha_n}$. Similarly, if $\rho = \lambda_1 \sigma_1 +
\cdots + \lambda_n \sigma_n$, for $n \in \bbN_+$, $\lambda_1,
\ldots, \lambda_n \in k \setminus \{ 0 \}$, and pairwise different
paths $\sigma_1$, \ldots, $\sigma_n$ in $\Delta$, is a relation,
then for a representation $M$ of $\Delta$ we put $M_\rho =
\lambda_1 M_{\sigma_1} + \cdots + \lambda_n M_{\sigma_n}$.

Let $\Lambda$ be the path algebra of a bound quiver $(\Delta,
\frakR)$. By $\rep (\Delta, \frakR)$ we denote the full
subcategory of the category of representations of $\Delta$
consisting of the representations $M$ such that $M_{\rho} = 0$ for
each $\rho \in \frakR$. The assignment which assigns to a
$\Lambda$-module $M$ the representation $(M_x, M_\alpha)_{x \in
\Delta_0, \alpha \in \Delta_1}$, where $M_x = x M$ for $x \in
\Delta_0$ and $M_\alpha (m) = \alpha m$ for $\alpha \in \Delta_1$
and $m \in M_{s_\alpha}$, induces an equivalence between $\mod
\Lambda$ and $\rep (\Delta, \frakR)$ (see for
example~\cite{AssemSimsonSkowronski2006}*{Theorem~III.1.6}). We
will treat this equivalence as identification. With every $M \in
\rep (\Delta, \frakR)$ (hence also with every $\Lambda$-module) we
associate its dimension vector $\bdim M \in \bbN^{\Delta_0}$
defined by $(\bdim M)_x = \dim_k M_x$ for $x \in \Delta_0$. We
call the elements of $\bbN^{\Delta_0}$ dimension vectors.

Let $\Lambda$ be an algebra of finite global dimension with the
Gabriel quiver $\Delta$. We define the homomological bilinear form
$\langle -, - \rangle_\Lambda : \bbZ^{\Delta_0} \times
\bbZ^{\Delta_0} \to \bbZ$ and the Euler quadratic form
$\chi_\Lambda : \bbZ^{\Delta_0} \to \bbZ$ by
\[
\langle \bd', \bd'' \rangle_\Lambda = \sum_{n \in \bbN} \sum_{x, y
\in \Delta_0} (-1)^n d_x' d_y'' \dim_k \Ext_\Lambda^n (S (x), S
(y))
\]
and $\chi_\Lambda (\bd) = \langle \bd, \bd \rangle_\Lambda$ for
$\bd', \bd'', \bd \in \bbZ^{\Delta_0}$, where for $x \in \Delta_0$
we put $P (x) = \Lambda x$ and $S (x) = P (x) / \rad P (x)$. It
follows by easy induction that
\[
\langle \bdim M, \bdim N \rangle_\Lambda = \sum_{n \in \bbN}
(-1)^n \dim_k \Ext_\Lambda^n (M, N)
\]
for all $\Lambda$-modules $M$ and $N$. Note that $P (x)$, $x \in
\Delta_0$, form a complete set of pairwise nonisomorphic
indecomposable projective $\Lambda$-modules.

Let $\Lambda$ be the path algebra of a bound quiver $(\Delta,
\frakR)$. By a convex subalgebra of $\Lambda$ we mean every
algebra of the form $k \Delta' / \langle \frakR' \rangle$, where
$\Delta'$ is a convex subquiver of $\Delta$ and $\frakR' = \frakR
\cap k \Delta'$. If $\Lambda'$ is a convex subalgebra of $\Lambda$
determined by a subquiver $\Delta'$ of $\Delta$, then we may
identify the $\Lambda'$-modules with the $\Lambda$-modules $M$
such that $\Hom_\Lambda (P (x), M) = 0$ for each $x \in \Delta_0
\setminus \Delta_0'$. Moreover, $\Ext_{\Lambda'}^n (M, N) =
\Ext_\Lambda^n (M, N)$ for all $\Lambda'$-modules $M$ and $N$ and
$n \in \bbN$. Consequently, $\gldim \Lambda' < \infty$ provided
$\gldim \Lambda < \infty$ and, if this is the case, $\langle \bd',
\bd'' \rangle_{\Lambda'} = \langle \bd', \bd'' \rangle_\Lambda$
for all $\bd', \bd'' \in \bbZ^{\Delta_0'}$.

An algebra $\Lambda$ is called shod if either $\pd_\Lambda X \leq
1$ or $\id_\Lambda X \leq 1$ for each indecomposable
$\Lambda$-module $X$. This condition implies in particular that
$\gldim \Lambda \leq
3$~\cite{HappelReitenSmalo1996a}*{Proposition~II.1.1}. A shod
algebra $\Lambda$ is called strict shod if $\gldim \Lambda = 3$.

Let $\Lambda$ be a strict shod algebra with the Gabriel quiver
$\Delta$. By $\calL_\Lambda$ we denote the class of indecomposable
$\Lambda$-modules $X$ such that $\pd_\Lambda Y \leq 1$ for each
indecomposable $\Lambda$-module $Y$ such that there exists a
sequence $Y = X_0, \ldots, X_n = X$ such that $\Hom_\Lambda (X_{i
- 1}, X_i) \neq 0$ for each $i \in [1, n]$. Obviously,
$\pd_\Lambda X \leq 1$ for each $X \in \calL_\Lambda$. An
indecomposable $\Lambda$-module $X$ if called $\Ext$-injective in
$\calL_\Lambda$ if $X \in \calL_\Lambda$ and $\Ext_\Lambda^1 (Y,
X) = 0$ for each $Y \in \calL_\Lambda$. By $J_\Lambda$ we denote
the direct sum of a complete set of pairwise nonisomorphic
$\Ext$-injective modules in $\calL_\Lambda$. Let $\Delta'$ denote
the full subquiver of $\Delta$ such that $x \in \Delta_0'$ if and
only if $P (x) \in \calL_\Lambda$. Then $\Delta'$ is a convex
subquiver of $\Delta$ and we denote by $\Lambda_\lambda$ the
convex subalgebra of $\Lambda$ determined $\Delta'$. Dually, we
define $\calR_\Lambda$, $Q_\Lambda$ and $\Lambda_\rho$, and we put
$\calP_\Lambda = \calR_\Lambda \setminus \calL_\Lambda$. We know
from~\cite{CoelhoLanzilotta1999}*{Theorem} that if $X$ is an
indecomposable $\Lambda$-module, then either $X \in \calL_\Lambda$
or $X \in \calR_\Lambda$.

For a module $M$ over an algebra $\Lambda$ let $\add M$ denotes
the full subcategory of $\mod \Lambda$ formed by all direct sums
of direct summands of $M$. We define $\add \calC$ for a class
$\calC$ of $\Lambda$-modules similarly. The following proposition
collects the properties of the strict shod algebras which will be
needed in our proofs.

\begin{proposition} \label{proposition_shod}
Let $\Lambda$ be a strict shod algebra.
\renewcommand{\labelenumi}{(\theenumi)}
\begin{enumerate}

\item \label{claim1}
$\Hom_\Lambda (Y, X) = 0$ and $\Ext_\Lambda^n (X, Y) = 0$ for all
$n \in \bbN_+$, $X \in \calL_\Lambda$, and $Y \in \calP_\Lambda$.

\item \label{claim2}
If $\Hom_\Lambda (J_\Lambda \oplus Q_\Lambda, X) \neq 0 \neq
\Hom_\Lambda (X, J_\Lambda \oplus Q_\Lambda)$ for an
indecomposable $\Lambda$-module $X$, then $X \in \add (J_\Lambda
\oplus Q_\Lambda)$.

\item \label{claim3}
If $X \in \add (J_\Lambda \oplus Q_\Lambda)$, then $\End_\Lambda
(X) \simeq k$ and $\Ext_\Lambda^n (X, X) = 0$ for each $n \in
\bbN_+$.

\item \label{claim4}
$\gldim \Lambda_\lambda \leq 2$ and $\gldim \Lambda_\rho \leq 2$.

\item \label{claim5}
If $M \in \mod \Lambda$, then $M \in \add \calL_\Lambda$ if and
only if $\Ext_\Lambda^1 (M, J_\Lambda) = 0$ and $\Hom_\Lambda (P,
M) = 0$ for each indecomposable projective $\Lambda$-module $P$
such that $P \not \in \calL_\Lambda$.

\item \label{claim6}
If $M \in \mod \Lambda$, then $M \in \add \calP_\Lambda$ if and
only if $\Hom_\Lambda (M, J_\Lambda) = 0$.

\end{enumerate}
\end{proposition}

\begin{proof}
\eqref{claim1}~Obvious (one uses Auslander--Reiten formula, see
for example~\cite{ AssemSimsonSkowronski2006}*{Theorem~2.13}).

\eqref{claim2}~It follows
from~\cite{AssemCoelhoTrepode2004}*{Proposition~3.4} and its dual.

\eqref{claim3}~The first part follows for example from~\cite{
CoelhoHappelUnger2002}*{Proposition~2.4}, the rest is clear.

\eqref{claim4}~See~\cite{Skowronski2003}*{Lemma~3.1}.

\eqref{claim5}~The latter part of the condition says that every $X
\in \calL_\Lambda$ is a $\Lambda_\lambda$-module, which is
obvious. Now the claim follows from the dual
of~\cite{AssemSimsonSkowronski2006}*{Theorem~VI.2.5}, since
$J_\Lambda$ is a cotilting $\calL_\Lambda$-module and $\add
\calL_\Lambda$ coincides with the class of $\Lambda$-modules
(equivalently, $\Lambda_\lambda$-modules) cogenerated by
$J_\Lambda$~\cite{AssemCoelhoTrepode2004}*{Theorem~4.2}.

\eqref{claim6}~Since $X \in \calP_\Lambda$ if and only if $X \not
\in \calL_\Lambda$, this follows
from~\cite{AssemCoelhoTrepode2004}*{Theorem~4.2}.
\end{proof}

\section{Preliminaries on module varieties}

Let $\Lambda$ be the path algebra of a bound quiver $(\Delta,
\frakR)$ and $\bd \in \bbN^{\Delta_0}$. By the variety of
$\Lambda$-modules of dimension vector $\bd$ we mean the set
$\mod_\Lambda (\bd)$ consisting of all $M \in \rep (\Delta,
\frakR)$ such that $M_x = k^{d_x}$. By forgetting the spaces
$M_x$, $x \in \Delta$, we identify $\mod_\Lambda (\bd)$ with a
Zariski-closed subset of $\prod_{\alpha \in \Delta_1}
\bbM_{d_{t_\alpha} \times d_{s_\alpha}} (k)$. Observe that for
each $\Lambda$-module $N$ of dimension vector $\bd$ there exists
$M \in \mod_\Lambda (\bd)$ such that $M \simeq N$. We will usually
assume that all considered $\Lambda$-modules are points of module
varieties.

Let $\Lambda$ be an algebra, $\Delta$ its Gabriel quiver, and $\bd
\in \bbN^{\Delta_0}$. Then $\GL (\bd) = \prod_{x \in \Delta_0}
\GL_{d_x} (k)$ acts on $\mod_\Lambda (\bd)$ by conjugation: $(g
M)_{\alpha} = g_{t_\alpha} M_\alpha g_{s_\alpha}^{-1}$ for $g \in
\GL (\bd)$, $M \in \mod_\Lambda (\bd)$, and $\alpha \in \Delta_1$.
Observe that $M \simeq N$ for $M, N \in \mod_\Lambda (\bd)$ if and
only if $\calO (M) = \calO (N)$, where for $M \in \mod_\Lambda
(\bd)$ we denote by $\calO (M)$ its $\GL (\bd)$-orbit. One easily
calculates that $\dim \calO (M) = \dim \GL (\bd) - \dim_k
\End_\Lambda (M)$ for each $M \in \mod_\Lambda
(\bd)$~\cite{KraftRiedtmann1986}*{2.2}. Moreover, we have a
canonical injection
\begin{equation} \label{eq_inj}
T_M \mod_\Lambda (\bd) / T_M \calO (M) \hookrightarrow
\Ext_\Lambda^1 (M, M)
\end{equation}
where for a point $x$ of a variety $\calV$ we denote by $T_x
\calV$ the tangent space to $\calV$ at
$x$~\cite{Kraft1984}*{II.2.7, Satz~4}. Consequently,
\begin{equation} \label{eq_dim}
\dim_k T_M \mod_\Lambda (\bd) \leq \dim \GL (\bd) - \dim_k
\End_\Lambda (M) + \dim_k \Ext_\Lambda^1 (M, M).
\end{equation}

Let $M, N \in \mod_\Lambda (\bd)$ for $\bd \in \bbN^{\Delta_0}$.
We call $N$ a degeneration of $M$ and write $M \leq_{\deg} N$ if
$N \in \ol{\calO (M)}$. If $M \leq_{\deg} N$ and $M \not \simeq
N$, then we write $M <_{\deg} N$ and call $N$ a proper
degeneration of $M$. A degeneration $M \leq_{\deg} N$ is called
minimal if $M <_{\deg} N$ and there is no $L \in \mod_\Lambda
(\bd)$ such that $M <_{\deg} L <_{\deg} N$. The formula for the
dimension of an orbit implies that $\dim_k \End_\Lambda (M) <
\dim_k \End_\Lambda (N)$ provided $M <_{\deg} N$. More generally,
if $M \leq_{\deg} N$, then $\dim_k \Hom_\Lambda (M, X) \leq \dim_k
\Hom_\Lambda (N, X)$ and $\dim_k \Hom_\Lambda (X, M) \leq \dim_k
\Hom_\Lambda (X, N)$ for all $\Lambda$-modules
$X$~\cite{Riedtmann1986b}*{Proposition~2.1}. On the other hand, if
there exists an exact sequence $0 \to N_1 \to M \to N_2 \to 0$
such that $N_1 \oplus N_2 \simeq N$, then $M \leq_{\deg} N$ (see
for example~\cite{Bongartz1996}*{Lemma~1.1}).

Let $\calV$ be a variety and $x \in \calV$. We say that $x$ is a
regular point of $\calV$ if $\dim_k T_x \calV = \dim_x \calV$
(equivalently, $\dim_k T_x \calV \leq \dim_x \calV$), where
$\dim_x \calV$ denotes the maximum of the dimensions of the
irreducible components of $\calV$ passing through $x$. In
particular, if $\calV$ is irreducible, then $x$ is a regular point
of $\calV$ if and only if $\dim_k T_x \calV = \dim \calV$
(equivalently, $\dim_k T_x \calV \leq \dim \calV$). We say that
$\calV$ is regular in codimension one if the codimension of the
complement of the set of regular points of $\calV$ is at least
$2$.

Now we show a special case of a more general fact proved by
Gei\ss~\cite{Geiss1996a}. We start with the following lemma, where
for a dimension vector $\bd$ we put $a_\Lambda (\bd) = \dim \GL
(\bd) - \chi_\Lambda (\bd)$.

\begin{lemma} \label{lemma_tangent}
Let $\Lambda$ be an algebra of finite global dimension. If $N$ is
a $\Lambda$-module of dimension vector $\bd$ such that
$\Ext_\Lambda^n (N, N) = 0$ for each $n \in \bbN_+$ with $n \geq
2$, then $\dim T_N \mod_\Lambda (\bd) \leq a_\Lambda (\bd)$.
\end{lemma}

\begin{proof}
It follows immediately from~\eqref{eq_dim}.
\end{proof}

Observe that if $\Lambda$ is an algebra of finite global dimension
and $M$ a $\Lambda$-module such that $\Ext_\Lambda^n (M, M) = 0$
for each $n \in \bbN_+$, then $\dim \calO (M) = a_\Lambda (\bd)$.
Together with the above lemma this immediately implies the
following.

\begin{corollary} \label{corollary_regular}
Let $\Lambda$ be an algebra of finite global dimension, and $M$
and $N$ $\Lambda$-modules such that $\Ext_\Lambda^n (M, M) = 0$
for each $n \in \bbN_+$ and $\Ext_\Lambda^n (N, N) = 0$ for each
$n \in \bbN_+$ with $n \geq 2$. If $M \leq_{\deg} N$, then $N$ is
a regular point of $\ol{\calO (M)}$. \qed
\end{corollary}

Let $\Lambda$ be an algebra and $\bd'$, $\bd''$ dimension vectors.
For $d \in \bbN$ we denote by $\calE_d^{\bd', \bd''}$ the set of
all $(U, V) \in \mod_\Lambda (\bd') \times \mod_\Lambda (\bd'')$
such that $\Hom_\Lambda (V, U) = 0$ and $\dim_k \Ext_\Lambda^1 (V,
U) = d$. By extending the arguments
from~\cite{Bobinski2008preprint}*{(2.7)} in an obvious way one
shows that if $\gldim \Lambda < \infty$ and $(U, V) \in
\calE_d^{\bd', \bd''}$, then
\begin{equation} \label{eq_calE}
\dim_k T_{U, V} \calE_d^{\bd', \bd''} \leq a_\Lambda (\bd') +
a_\Lambda (\bd'') - \dim_k \Ext_\Lambda^2 (V, U)
\end{equation}
provided the following conditions are satisfied: $\Ext_\Lambda^n
(U, U) = 0 = \Ext_\Lambda^n (V, V)$ for each $n \in \bbN_+$ with
$n \geq 2$ and there exists an exact sequence $0 \to U \to M \to V
\to 0$ such that either $\Ext_\Lambda^2 (V, M) = 0$ or
$\Ext_\Lambda^2 (M, U) = 0$.

For the rest of the section we assume that $\Lambda$ is a strict
shod algebra, $\calL = \calL_\Lambda$, and $\calP =
\calP_\Lambda$. For a dimension vector $\bd$ we denote by $\calL
(\bd)$ the set of all $M \in \mod_\Lambda (\bd)$ such that $M \in
\add \calL$. We define $\calP (\bd)$ similarly. More generally, if
$\bd'$ and $\bd''$ are dimensions vectors, then $\calL (\bd')
\oplus \calP (\bd'')$ denotes the set of all $M \in \mod_\Lambda
(\bd' + \bd'')$ such that $M \simeq U \oplus V$ for some $U \in
\calL (\bd')$ and $V \in \calP (\bd'')$.

Let $\bd$ be a dimension vector. It follows from
Proposition~\ref{proposition_shod}~\eqref{claim5}
and~\eqref{claim6} that $\calL (\bd)$ and $\calP (\bd)$ are open
subsets of $\mod_\Lambda (\bd)$. Moreover, if $\calL (\bd) \neq
\varnothing$, then it is irreducible of dimension $a_\Lambda
(\bd)$~\cite{BarotSchroer2001}*{Proposition~3.1} (for the last
statement one uses that $a_\Lambda (\bd) = a_{\Lambda_\lambda}
(\bd)$ and Proposition~\ref{proposition_shod}~\eqref{claim4}).
This implies that $\dim_k T_M \mod_\Lambda (\bd) = a_\Lambda
(\bd)$ for each $M \in \calL (\bd)$, according to
Lemma~\ref{lemma_tangent}. The analogous statement holds for
$\calP (\bd)$. Finally, if $\bd'$ and $\bd''$ are dimension
vectors such that $\calL (\bd')$ and $\calP (\bd'')$ are nonempty,
then $\calL (\bd') \oplus \calP (\bd'')$ is an irreducible
constructible subset of $\mod_\Lambda (\bd' + \bd'')$.

\section{Proof of the main result}

This section is devoted for proving the main result of the paper.
The line of the proof mainly follows that of the proof
of~\cite{Bobinski2007preprint}*{Main Theorem}, however some
arguments have to be adapted to the setting of strict shod
algebras.

Throughout this section $\Lambda$ is a fixed strict shod algebra,
$T$ the canonical tilting $\Lambda$-module, $\calL =
\calL_\Lambda$, $J = J_\Lambda$, $\calP = \calP_\Lambda$, and $Q =
Q_\Lambda$. Finally, we fix a $\Lambda$-module $M \in \add T$.
Observe that $M \in \add (J \oplus Q)$. We write $M = L \oplus R$
with $L \in \add \calL$ and $R \in \add \calR$, and put $\bd =
\bdim M$.

The main aim of this section is to prove the following extension
of Main Theorem.

\begin{theorem} \label{theorem}
$\ol{\calO (M)}$ is regular in codimension one.
\end{theorem}

We start with the following.

\begin{lemma}
If $N$ is a minimal degeneration of $M$, then there exists an
exact sequence $0 \to U \to M \to V \to 0$ such that $N \simeq U
\oplus V$.
\end{lemma}

\begin{proof}
If there is not such a sequence, then the minimality of the
degeneration $M <_{\deg} N$ and~\cite{Zwara2000}*{Theorem~4} imply
that there exists an indecomposable direct summand $X$ of $N$ such
that $M' <_{\deg} X$ for a direct summand $M'$ of $M$. Then
\begin{align*}
\dim_k \Hom_\Lambda (M', X) \geq \dim_k \Hom_\Lambda (M', M') > 0
\\
\intertext{and} %
\dim_k \Hom_\Lambda (X, M') \geq \dim_k \Hom_\Lambda (M', M') > 0,
\end{align*}
hence $X \in \add (J \oplus Q)$ by
Proposition~\ref{proposition_shod}~\eqref{claim2}. Thus
$\End_\Lambda (X) \simeq k$ according
to~Proposition~\ref{proposition_shod}~\eqref{claim3}, and $\dim_k
\End_\Lambda (M') = 0$, a contradiction.
\end{proof}

\begin{proposition} \label{prop_min}
Let $U \in \add \calL$ and $V \in \add \calP$ be such that $U
\oplus V$ is a minimal degeneration of $M$.
\renewcommand{\labelenumi}{(\theenumi)}
\begin{enumerate}

\item \label{point1}
If $\bdim U \neq \bdim L$, then there exists an exact sequence $0
\to U \to M \to V \to 0$.

\item \label{point2}
If $\bdim U = \bdim L$, then either $U \simeq L$ or $V \simeq R$.

\end{enumerate}
\end{proposition}

\begin{proof}
\eqref{point1}~Assume that $\bdim U \neq \bdim L$. According to
the above lemma there exists an exact sequence $\sigma : 0 \to U'
\to M \to V' \to 0$ such that $U' \oplus V' \simeq U \oplus V$.
Write $U' = L_1 \oplus R_1$ and $V' = L_2 \oplus R_2$, where $U
\simeq L_1 \oplus L_2$ and $V \simeq R_1 \oplus R_2$. We first
show that we may assume that either $R_1 = 0$ or $L_2 = 0$.

Let $p : U' \to R_1$ be the canonical projection and $\sigma' : 0
\to R_1 \to M' \to V' \to 0$ be the push-out exact sequence $p
\circ \sigma$. If $\sigma'$ splits, then we easily get an exact
sequence $0 \to L_1 \to M \to V' \oplus R_1 \to 0$, thus the claim
follows in this case. On the other hand, if $\sigma'$ does not
split, then $M \simeq M' \oplus L_1$ by the minimality of the
degeneration $M <_{\deg} U \oplus V$. Moreover, the pull-back
sequence exact $\sigma' \circ i$, where $i : L_2 \to V'$ is the
canonical injection, splits by
Proposition~\ref{proposition_shod}~\eqref{claim1} and we get an
exact sequence $0 \to L_2 \oplus R_1 \to M' \to R_2 \to 0$, which
in turn gives rise to an exact sequence $0 \to U' \oplus L_2 \to M
\to R_2 \to 0$, as desired.

Without loss of generality we assume that $L_2 = 0$, i.e.\ $U
\simeq L_1$. If $\sigma'$ is as before, then $M' \in \add \calP$
(one may use for example
Proposition~\ref{proposition_shod}~\eqref{claim6}). Moreover, we
have a short exact sequence $0 \to U \to M \to M' \to 0$. This
sequence does not split, since $\bdim U \neq \bdim L$.
Consequently, $\sigma'$ splits, i.e.\ $M' \simeq V$, and we get an
exact sequence $0 \to U \to M \to V \to 0$, thus the claim
follows.

\eqref{point2}~Assume that $\bdim U = \bdim L$. Consequently,
$\bdim V = \bdim R$. Observe that both $U$ and $L$ are modules of
projective dimension at most $1$ and $\Ext_\Lambda^1 (L, L) = 0$.
Together with~\eqref{eq_inj} this implies that $\ol{\calO (L)}$ is
an irreducible component of $\mod_\Lambda (\bdim L)$, hence $L
\leq_{\deg} U$, since $\calL (\bdim L)$ is irreducible.
Analogously, $R \leq_{\deg} V$. If the both degenerations are
proper, then we have the following sequence of proper
degenerations
\[
M = L \oplus R <_{\deg} U \oplus R <_{\deg} U \oplus V
\]
which contradicts the minimality of the degeneration $M <_{\deg} U
\oplus V$.
\end{proof}

\begin{proof}[Proof of Theorem~\ref{theorem}]
Let $\calX$ be an irreducible component of $\ol{\calO (M)}
\setminus \calO (M)$. Our aim is to show that there exists an open
subset $\calU$ of $\calX$ such that all points of $\calU$ are
regular points of $\ol{\calO (M)}$. Observe that there exist
$\bd', \bd'' \in \bbN^{\Delta_0}$ such that $(\calL (\bd') \oplus
\calP (\bd'')) \cap \calX$ contains an open subset $\calU'$ of
$\calX$. Put
\begin{gather*}
d_0 = \min \{ \dim_k \End_\Lambda (N) \mid N \in \calX \}
\\
\intertext{and} %
d_1 = \min \{ \dim_k \Ext_\Lambda^1 (N, N) \mid N \in \calX \}.
\end{gather*}
Let $\calU$ be the set of all $N \in \calU'$ such that the
dimensions of $\End_\Lambda (N)$ and $\Ext_\Lambda^1 (N, N)$ equal
$d_0$ and $d_1$, respectively, and $N$ does not belong to an
irreducible component of $\ol{\calO (M)} \setminus \calO (M)$
different from $\calX$. Obviously, $\calU$ is a nonempty open
subset of $\calX$ consisting of minimal degenerations of $M$. Fix
$N \in \calU$, and $U \in \add \calL$ and $V \in \add \calR$ such
that $N \simeq U \oplus V$.

First assume that $\bd' = \bdim L$. According to
Proposition~\ref{prop_min}~\eqref{point2} we may assume without
loss of generality that $U \simeq L$. Moreover, since $\id_\Lambda
V \leq 1$, $V$ is a regular point of $\ol{\calO (R)}$ by
Corollary~\ref{corollary_regular}. Finally,
\begin{gather*}
\dim_k \Hom_\Lambda (L, R) = \langle \bd', \bd'' \rangle = \dim_k
\Hom_\Lambda (L, V)
\\
\intertext{and} %
\dim_k \Hom_\Lambda (R, L) = 0 = \dim_k \Hom_\Lambda (V, L)
\end{gather*}
hence the claim follows from~\cite{Bongartz1994}*{Theorem~2}.

Now assume that $\bd' \neq \bdim L$. According to
Proposition~\ref{prop_min}~\eqref{point1} there exists an exact
sequence $\xi : 0 \to U \to M \to V \to 0$, which in particular
implies that $\pd_\Lambda V \leq 2$. Put $\calV = \mod_\Lambda
(\bd') \times \mod_\Lambda (\bd'')$. We identify $\calV$ with a
subset of $\mod_\Lambda (\bd)$ in an obvious way. Without loss of
generality we may assume that $N \in \calV$, i.e.\ $N = U \oplus
V$. Since $\dim_k T_N \calV = a_\Lambda (\bd') + a_\Lambda
(\bd'')$, it is sufficient to show that
\renewcommand{\theequation}{$*$}
\begin{equation} \label{eq}
\dim_k T_N (\calV \cap \ol{\calO (M)}) \leq a_\Lambda (\bd') +
a_\Lambda (\bd'') - \dim_k \Ext_\Lambda^2 (V, U).
\end{equation}
Indeed, this will imply that
\[
\dim_k T_N \ol{\calO (M)} \leq \dim_k T_N \mod_\Lambda (\bd) -
\dim_k \Ext^2 (V, U),
\]
hence we may apply~\eqref{eq_dim} and use that $\Ext_\Lambda^2 (V,
U) = \Ext_\Lambda^2 (N, N)$ and $\pd_\Lambda N \leq 2$.

We prove~\eqref{eq}. Let $\calU_0$ be the intersection of $\calU$
with $\calV$. Observe that $\calU_0$ is an open subset of $\calV
\cap \ol{\calO (M)}$. Moreover, $\calU_0 \subset \calE_d^{\bd',
\bd''}$, where
\[
d = d_1 - d_0 + \chi_\Lambda (\bd') + \chi_\Lambda (\bd'') +
\langle \bd', \bd'' \rangle_\Lambda.
\]
Consequently, the claim follows from~\eqref{eq_calE}.
\end{proof}

\bibsection

\end{document}